\documentclass[a4paper,12pt]{article}
\usepackage{amssymb,amsmath,amsthm,latexsym}
\usepackage{amsfonts}
\usepackage{amsfonts}
\usepackage{graphicx}
\usepackage[pdftex,bookmarks,colorlinks=false]{hyperref}
\usepackage{verbatim}
\usepackage[font=small,labelfont=bf]{caption}
\usepackage{subcaption}

\newtheorem{theorem}{Theorem}[section]

\newtheorem{definition}[theorem]{Definition}

\newtheorem{proposition}[theorem]{Proposition}

\title{\bf The Sum and Product of Independence Numbers of Graphs and their Line Graphs}

\author{{\bf Susanth C \footnote{Department of Mathematics, Research \& Development Centre, Bharathiar University, Coimbatore - 641046, Tamilnadu, email: {\em susanth\_c@yahoo.com}}}   ~and
{\bf Sunny Joseph Kalayathankal\footnote{Department of Mathematics, Kuriakose Elias College, Mannanam, Kottayam - 686561, Kerala, email:{\em sunnyjoseph2000@yahoo.com}}}}
\date{}

\begin{document}
\maketitle
\begin{abstract}
The bounds on the sum and product of chromatic numbers of a graph and its complement are known as Nordhaus-Gaddum inequalities. 
In this paper, we study the bounds on the sum and product of the independence numbers of graphs and their line graphs. 
We also provide a new characterization of the certain graph classes.
\end{abstract}
\section{Introduction}
For all  terms and definitions, not defined specifically in this paper, we refer to \cite{FH}. Unless mentioned otherwise, all graphs considered here are simple, finite and have no isolated vertices.
\\Many problems in extremal graph theory seek the extreme values of graph parameters
on families of graphs. The classic paper of Nordhaus and Gaddum \cite{KCA} study the extreme values of the sum (or product) of a parameter on a graph and its complement, following  solving these problems for the chromatic number on n-vertex graphs. In this paper, we study such problems for some graphs and their associated graphs.
\begin{definition}\rm{
\cite{CGT} Two vertices that are not adjacent in a graph $G$
are said to be \textit{independent}. A set $S$ of vertices is independent if any two vertices of $S$ are independent. The \textit{vertex independence number} or simply the \textit{independence number}, of a graph $G$, denoted by  $\alpha(G)$ is the maximum cardinality among
the independent sets of vertices of $G$.}
\end{definition} 
\begin{definition}{\rm
\cite{BM1} A subset $M$ of the  edge set of $G$, is called a \textit{matching} in $G$ if no two of the edges in $M$ are adjacent. In other words, if for any two edges $e$ and $f$ in $M$, both the end vertices of $e$ are different from the end vertices of $f$.}
\end{definition}

\begin{definition}{\rm
\cite{BM1} A \textit{perfect matching} of a graph $G$ is a matching of $G$ containing $n/2$ edges, the largest possible, meaning perfect matchings are only possible on graphs with an even number of vertices.  A perfect matching sometimes called a \textit{complete matching} or \textit{1-factor}}.
\end{definition}

\begin{definition}{\rm
\cite{BM1} The matching number of a graph $G$, denoted by $\nu(G)$, is the size of a maximal independent edge set. It is also known as \textit{edge independence number}. The matching number $\nu(G)$  satisfies the inequality $\nu(G)\leq\lfloor \frac{n}{2}\rfloor $. \newline Equality occurs only for a perfect matching and graph $G$ has a perfect matching if and only if $|G|=2~\nu(G)$, where $|G|=n$ is the vertex count of $G$.}
\end{definition}

\begin{definition}{\rm
\cite{BM1}A \textit{maximum independent set} in a line graph corresponds to maximum matching in the original graph. } 
\end{definition}
In this paper, we discussed the sum and product of the independence numbers of certain class of graphs and their line graphs.


\section{New Results}
\begin{definition}{\rm
\cite{RJW} The line graph $L(G)$ of a simple graph $G$ is the graph whose vertices are in one-one
correspondence with the edges of $G$, two vertices of $L(G)$ being adjacent if and only if
the corresponding edges of $G$ are adjacent.}
\end{definition}

\begin{theorem}\label{th1}{\rm
\cite{BB} The independence number of the line graph of a graph $G$ is equal to the matching number of $G$.}
\end{theorem}
\begin{proposition}{\rm 
The sum of the independence number of a complete graph and its line graph is $\lfloor \frac{n}{2}\rfloor +1$ and their product is  $\lfloor \frac{n}{2}\rfloor$.}
\end{proposition}
\begin{proof}
The independence number of a complete graph $K_n$ on $n$ vertices is $1$, since each vertex is joined with every other vertex of the graph $G$.  
By theorem \ref{th1}, the independence number of the line graph of $K_n$ is the matching number of $K_n = \lfloor \frac{n}{2}\rfloor $.\newline Therefore
\begin{center}
$\alpha(K_n)+\alpha(L(K_n))=\lfloor \frac{n}{2}\rfloor+1$ and 
\end{center}  
\begin{center}
$\alpha(K_n).\alpha(L(K_n))=\lfloor \frac{n}{2}\rfloor$.
\end{center}
\end{proof}

\begin{proposition}
For a bipartite graph $B_{m,n}$, 
\begin{center}
$\alpha(B_{m,n}) + \alpha(L(B_{m,n}))=m+n$  and
\end{center}
 \begin{center}
$\alpha(B_{m,n}) . \alpha(L(B_{m,n}))=mn$
 \end{center}
\end{proposition}
\begin{proof}
Without the loss of generality, let $m<n$.  The independence number of a bipartite graph, $\alpha(B_{m,n}) = max(m,n)=n$ and that of its line graph, $\alpha(L(B_{m,n}))$ = matching number of $B_{m,n}=\nu(B_{m,n})= min(m,n)=m$.\newline Therefore, 
\begin{center}
$\alpha(B_{m,n}) + \alpha(L(B_{m,n}))=m+n$  and
\end{center}
 \begin{center}
$\alpha(B_{m,n}) . \alpha(L(B_{m,n}))=mn$
 \end{center}
\end{proof}

\begin{definition}{\rm
\cite{FH} For $n\geq 3$, a \textit{wheel graph} $W_{n+1}$ is the graph $K_1+ C_{n}$.  A wheel graph $W_{n+1}$ has $n+1$ vertices and $2n$ edges.}
\end{definition}

\begin{theorem}{\rm
For $n\geq 3$, $\alpha(W_{n+1})+\alpha(L(W_{n+1}))= 2 \lfloor \frac{n}{2}\rfloor $ and $\alpha(W_{n+1}).\alpha(L(W_{n+1}))= (\lfloor \frac{n}{2}\rfloor)^2$.}
\end{theorem}
\begin{proof}
Let $I$ be the maximal independent set of a wheel graph $W_{n+1}$.  By definition a wheel graph is defined to be the graph $K_1+ C_{n}$.  If $K_1\in~I$, no vertex of $C_{n}$ can be in $I$.  Hence let $K_1\notin I$.  

Case - 1 : If $n$ is even, $C_{n}$ is an even cycle.  Then $C_{n}= v_1,v_2,....,v_{n},v_1$. Without loss of generality, choose $v_1$ to $I$.  Since $v_2$ is adjacent to $v_1$, $v_2\notin I$.  Now choose $v_3$ to $I$, since it is not adjacent to $v_1$.  Now $v_4$ cannot be selected to $I$, since it is adjacent to $v_3$.  Proceeding in this way, finite number of times, the vertices of the form $v_i, i=1,3,5,\ldots, n-1$ belong to $I$.  That is, $\alpha(C_n) = \frac{n}{2}$.

Case - 2 : If $n$ is odd, $C_{n}$ is an odd cycle.  Then $C_{n}= v_1,v_2,....,v_{n},v_1$. Without loss of generality, choose $v_1$ to $I$.  Since $v_2$ is adjacent to $v_1$, $v_2\notin I$.  Now choose $v_3$ to $I$, since it is not adjacent to $v_1$.  Now $v_4$ cannot be selected to $I$, since it is adjacent to $v_3$.  Proceeding in this way, finite number of times, the vertices of the form $v_i, i=1,3,5,\ldots, n-2$ belong to $I$.  That is, $\alpha(C_n) =  \frac{n-1}{2}$.
\\From the above two cases, it is clear that the independence number of a wheel graph $W_{n+1}$    is $\lfloor \frac{n}{2}\rfloor$.

Now, consider the line graph of the wheel graph $W_{n+1}$.  By theorem \ref{th1}, the independence number of $L(W_{n+1})$ is equal to the matching number of $W_{n+1}$.

Let $M$ be a maximal matching set of the wheel graph $W_{n+1}$.  Then $\alpha (L(W_{n+1}))=\nu(W_{n+1})$.  Let $e_1,e_2,....,e_{n}$ be the edges the outer cycle taken in the order of the wheel graph $W_{n+1}$ and let $e_1',e_2',....,e_{n}'$ be the edges incident on the vertex of $K_1$.  

Case - 1 : If $n$ is even, without loss of generality choose $e_1$ to the set $M$.  Now, $e_2\notin M$ as $e_2$ is adjacent to $e_1$.  Now take the edge $e_3$ to $M$.  Since $e_4$ is adjacent to $e_3$,  $e_4\notin M$.  Proceeding in this manner, a finite number of times, an edge of the form $e_i, i=1,3,5,\ldots, n-1$ belong to $M$. In this case  no edge of the form $e_j'$ can be a member of $M$.  That is, $|M|= \frac{n}{2}$. 

Case - 2 : If $n$ is odd, an edge of the form $e_i, i=1,3,5,\ldots, n-2$ belong to $M$.  Moreover there is one edge $e_j'$ that is incident on $K_1$ and is not adjacent to any of the edges in $M$.  That is, $|M|= \frac{n-1}{2}$.  

From the above two cases, we follow that $\alpha(L(W_{n+1}))= \lfloor \frac{n}{2}\rfloor $

Therefore, $\alpha(W_{n+1})+\alpha(L(W_{n+1}))= 2 \lfloor \frac{n}{2}\rfloor $ and $\alpha(W_{n+1}).\alpha(L(W_{n+1}))= (\lfloor \frac{n}{2}\rfloor)^2$.
\end{proof}

\begin{definition}{\rm
\cite{JAG} \textit{Helm graphs} are graphs obtained from a wheel by attaching one pendant edge to each vertex of the cycle.}
\end{definition}

\begin{theorem}
For a helm graph $H_n$, $n\geq 3$,
$\alpha(H_n)+\alpha(L(H_n))= 2n+1$ and
$\alpha(H_n).\alpha(L(H_n))= n(n+1)$.
\end{theorem}

\begin{proof}
Let $I$ be a maximal independent set of a helm graph $H_n$.  Then, its elements are the set of all pendent vertices together with the vertex of $K_1$.  So $I$ consists of $n+1$ elements.  There fore, $\alpha(H_n)=n+1$.

Now consider the line graph of the helm graph $H_{n}$.  By theorem \ref{th1}, the independence number of $L(H_{n})$ is equal to the matching number of $H_{n}$.  Let $M$ be a maximal matching set of the helm graph $H_{n}$.  Then $\alpha(L(H_{n}))=\nu(H_{n})$.  Let $e_1,e_2,....,e_{n}$ be the pendent edges incident with the outer cycle taken in order of the helm graph $H_{n}$.  Then if we take these $n$ edges to $M$, it will be a maximum matching in $H_n$.  This means the matching number of a helm graph, $\nu(H_{n}) = n = \alpha(L(H_{n}))$.

Therefore, 

$\alpha(H_n)+\alpha(L(H_n))= 2n+1$ and
$\alpha(H_n).\alpha(L(H_n))= n(n+1)$.
\end{proof}

\begin{definition}{\rm
\cite{DBW} Given a vertex $x$ and a set $U$ of vertices, an $x$, $U-$fan is a set of paths from $x$ to $U$ such that any two of them share only the vertex $x$.  A $U-$fan is denoted by $F_{1,n}$.}
\end{definition}

\begin{theorem}
For a fan graph $F_{1,n}$, 
\[\alpha(F_{1,n})+\alpha(L(F_{1,n}))=
\left\{
\begin{array}{ll}
n & \text{if $n$ is even} \\
n+1 & \text{if $n$ is odd}
\end{array}\right.\]
\[\alpha(F_{1,n}).\alpha(L(F_{1,n}))=
\left\{
\begin{array}{ll}
\frac{n^2}{4} & \text{if $n$ is even} \\
\frac{(n+1)^2}{4} & \text{if $n$ is odd}
\end{array}\right.\]
\end{theorem}

\begin{proof}
Let $I$ be a maximal independent set of a fan graph $F_{1,n}$.  By the definition a fan graph is defined to be the graph $K_1+ P_{n}$.  If $K_1\in~I$, no vertex of $P_{n}$ can be in $I$.  Hence let $K_1\notin I$.

Case-1: If $n$ is even, then $P_{n}$ is an odd path.  Then, $P_{n}= v_1,v_2,....,v_{n}$. Without loss of generality, choose $v_1$ to $I$.  Since $v_2$ is adjacent to $v_1$, $v_2\notin I$.  Now choose $v_3$ to $I$, since it is not adjacent to $v_1$.  Now $v_4$ cannot be selected to $I$, since it is adjacent to $v_3$.  Proceeding in this way, finite number of times, the vertices of the form $v_i, i=1,3,5,\ldots, n-1$ belong to $I$.  That is, $\alpha(P_n) = \frac{n}{2}$.

Case-2 : if $n$ is odd, then $P_{n}$ is an even path.  Then, $P_{n}= v_1,v_2,....,v_{n}$. Without loss of generality, choose $v_1$ to $I$.  Since $v_2$ is adjacent to $v_1$, $v_2\notin I$.  Now choose $v_3$ to $I$, since it is not adjacent to $v_1$.  Now, $v_4$ cannot be selected to $I$, since it is adjacent to $v_3$.  Proceeding in this way, finite number of times, the vertices of the form $v_i, i=1,3,5,\ldots, n$ belong to $I$.  That is, $\alpha(P_n) =  \frac{n+1}{2}$.
\\From the above two cases it is clear that the independence number of a fan graph $F_{1,n}$    is either $\frac{n}{2}$ or $\frac{n+1}{2}$, depending on $n$ is even or odd.

Now consider the line graph of the fan graph $F_{1,n}$.  By theorem \ref{th1}, the independence number of $L(F_{1,n})$ is equal to the matching number of $F_{1,n}$.

Let $M$ be a maximal matching set of the fan graph $F_{1,n}$.  Then $\alpha(L(F_{1,n}))=\nu(F_{1,n})$.  Let $e_1,e_2,....,e_{n-1}$ be the edges the outer path taken in order of the fan graph $F_{1,n}$ and let $e_1',e_2',....,e_{n}'$ be the edges incident with the vertex of $K_1$.  

Case - 1 : Without loss of generality choose $e_1$ to the set $M$.  Now, since $e_2$ is adjacent to $e_1$,  $e_2\notin M$.  Now take the edge $e_3$ to $M$.  Since $e_4$ is adjacent to $e_3$,  $e_4\notin M$.  Proceeding in this manner, a finite number of times, an edge of the form $e_i, i=1,3,5,\ldots, n-1$ belong to $M$. In this case  no edge of the form $e_j'$ can be a member of $M$.  That is,$|M|= \frac{n}{2}$. 

Case - 2 : If $n$ is odd, an edge $e_i, i=1,3,5,\ldots, n-2$ belongs to $M$.  Moreover there is one edge $e_j'$ that is incident on $K_1$ and is not adjacent to any of the edges in $M$.  That is, $|M|= \frac{n+1}{2}$.  

From the above two cases, we follow that $\alpha(L(F_{1,n}))= \nu(F_{1,n})$ is either $\frac{n}{2}$ or $\frac{n+1}{2}$ depending on $n$ is even or odd.

Therefore
For a fan graph $F_{1,n}$, 
\[\alpha(F_{1,n})+\alpha(L(F_{1,n}))=
\left\{
\begin{array}{ll}
n & \text{if $n$ is even} \\
n+1 & \text{if $n$ is odd}
\end{array}\right.\]
\[\alpha(F_{1,n}).\alpha(L(F_{1,n}))=
\left\{
\begin{array}{ll}
\frac{n^2}{4} & \text{if $n$ is even} \\
\frac{(n+1)^2}{4} & \text{if $n$ is odd}
\end{array}\right.\]

\end{proof}
\begin{definition}{\rm
\cite{AVJ, IS} An $n-$sun or a \textit{trampoline}, denoted by $S_n$, is a chordal graph on $2n$ vertices, where $n \geq 3$, whose vertex set can be partitioned into two sets $U = \{u_1, u_2, u_3, . . . , u_n\}$ and $W = \{w_1, w_2, w_3, . . . , w_n\}$ such that $U$ is an independent set of $G$ and $u_i$ is adjacent to $w_j$ if and only if $j = i$ or $j = i + 1 (\mod{n})$. A \textit{complete sun} is a sun $G$ where the induced subgraph  $\langle U\rangle$ is complete.}
\end{definition}
\begin{theorem}{\rm
For a sun graph $S_n$, $n\geq 3$,
$\alpha(S_n)+\alpha(L(S_n))= 2n$ and
$\alpha(S_n).\alpha(L(S_n))= n^2$.}
\end{theorem}
\begin{proof}{\rm
Let $S_n$ be a sun graph on $2n$ vertices.  Let $V = \{v_1, v_2, v_3, . . . , v_n\}$ be the vertex set of $K_n$ and $U = \{u_1, u_2, u_3, . . . , u_n\}$ be the set of vertices attached to the edges of the outer ring of $K_n$.  Clearly, all the vertices in $U$ are independent and $U$ itself is the maximum independent set in $S_n$.  Therefore the independence number of $S_n$, $\alpha(S_n)=n$.  Let  $E_1 = \{e_1, e_2, e_3, . . . , e_n\}$ be the edge set of the outer rings of $K_n$.  Now, corresponding to each edge of the outer ring of $K_n$, there exist two edges connecting its end vertices in $S_n$.  Therefore, let $E_2 = \{e_1', e_2', e_3', . . . , e_n'\}$ be the edge set in $S_n$ such that the pair $(e_j',e_k')$ corresponds to the edge $e_i$ of the outer rings of $K_n$.  Clearly, one edge among each pair $(e_j',e_k')$ contributes to a maximal matching of $L_n$.  That is, $\nu(S_n)=n$.  Therefore,  for a sun graph $S_n$, $n\geq 3$,
$\alpha(S_n)+\alpha(L(S_n))= 2n$ and
$\alpha(S_n).\alpha(L(S_n))= n^2$. }
\end{proof}
\begin{definition}{\rm
\cite{IS} The $n-$sunlet graph is the graph on $2n$ vertices obtained by attaching $n$ pendant edges to a cycle graph $C_n$ and is denoted by $L_n$.}
\end{definition}
\begin{theorem}{\rm
For a sunlet graph $L_n$ on $2n$ vertices, $n\geq 3$,
$\alpha(L_n)+\alpha(L(L_n))= 2n$ and
$\alpha(L_n).\alpha(L(L_n))= n^2$.}
\end{theorem}
\begin{proof}{\rm
Let $L_n$ be a sunlet graph on $2n$ vertices.    Let $V = \{v_1, v_2, v_3, . . . , v_n\}$ be the vertex set of the cycle $C_n$ and $U = \{u_1, u_2, u_3, . . . , u_n\}$ be the set of pendent vertices attached to the vertices of the cycle $C_n$.  Clearly, all the vertices in $U$ are independent and $U$ itself is the maximum independent set in $L_n$.  Therefore, $\alpha(L_n)=n$.  Let  $E_1 = \{e_1, e_2, e_3, . . . , e_n\}$ be the edge set of $C_n$.  Now, corresponding to each edge of $C_n$, there exist two edges connecting its end vertices in $L_n$.  Therefore, let $E_2 = \{e_1', e_2', e_3', . . . , e_n'\}$ be the edge set in $L_n$ such that the pair $(e_j',e_k')$ corresponds to the edge $e_i$ of the cycle $C_n$.  Clearly, the set $E_2 = \{e_1', e_2', e_3', . . . , e_n'\}$ contributes to a maximal matching of $L_n$.  That is, $\alpha(L(L_n))=\nu(L_n)=n$.  Therefore,  for a sunlet graph $L_n$ on $2n$ vertices, $n\geq 3$,
$\alpha(L_n)+\alpha(L(L_n))= 2n$ and
$\alpha(L_n).\alpha(L(L_n))= n^2$. }
\end{proof}
\begin{definition}{\rm
\cite{FH} The \textit{armed crown} is a graph $G$ obtained by adjoining a path $P_{m}$ to every vertex of a cycle $C_n$.}
\end{definition}
\begin{theorem}{\rm
For an armed crown graph $G$ with a path $P_m$ and a cycle $C_n$, 
\[\alpha(G)+\alpha(L(G))=
\left\{\begin{array}{ll}
\lfloor\frac{n}{2}\rfloor \lbrack \frac{m+1}{2}+1\rbrack + \lbrack\frac{m-1}{2}\rbrack \lbrack n+\lceil\frac{n}{2}\rceil\rbrack  & \text{; if $m$, $n$ are odd} 
\\mn & \text{; otherwise}
\end{array}\right.\]\\
\[\alpha(G).\alpha(L(G))=
\left\{\begin{array}{ll}
\lbrack\lfloor\frac{n}{2}\rfloor\lbrack\frac{m+1}{2}\rbrack+\lceil\frac{n}{2}\rceil\lbrack\frac{m-1}{2}\rbrack\rbrack.\lbrack\lfloor\frac{n}{2}\rfloor+n\lbrack\frac{m-1}{2}\rbrack\rbrack
& \text{; if $m$, $n$ are odd} 
\\\frac{n^2m^2}{4} & \text{; otherwise}
\end{array}\right.\]}
\end{theorem}

\begin{proof}
Note that the number of vertices of $P_m$ is $m+1$. Let $u_1,u_2,..., u_n$ be the vertices of the cycle $C_n$.  Let $u_i^1, u_i^2,....u_i^m$ be the vertices of the paths of length $m$ attached with $u_i$, $1\leq i \leq n$ with identification of $u_i$ and $u_i^m$.
\\\\Case-1: (when $m$ is even and $n$ is even) 
\\
Since $u_1^1, u_1^2,....u_1^m$ be the vertices of the first path attached to the first vertex $u_1$ of the cycle $C_n$, the maximal independent set consists of exactly $\frac{m}{2}$ elements.  Since there are $n$ number of paths attached to every vertex $u_i$, $1\leq i\leq n$ of the cycle $C_n$, which contributes $n\frac{m}{2}$ number of elements to the maximal independent set $I$. Therefore, $\alpha(G)=\frac{nm}{2}$.  Let $E_1=\{ e_1, e_2, ...e_n\}$ be the edge set of the cycle $C_n$ and let $E_2=\{e_i^1, e_i^2, ... e_i^m\}$ be the edge set of the path $P_m$.  Now for every $u_i$ in $C_n$, there exists a path with $m$ number of edges.  Clearly, $\frac{m}{2}$ edges contributes to a maximal matching of $P_m$.  For each $n$ vertices $u_i$, $(1\leq i \leq n)$  of $C_n$, there is a path $P_m$ adjoined to it, so that the maximal matching in $G$, $\nu(G)=\frac{nm}{2}$.
\\Therefore, $\alpha(G)+\alpha(L(G))= \frac{nm}{2}+ \frac{nm}{2} = nm$ and
$\alpha(G).\alpha(L(G))= \frac{nm}{2}. \frac{nm}{2} = \frac{n^2m^2}{4}$.
\\ 
\\Case - 2 : ($m$ is odd and $n$ is even)
\\ Since $m$ is odd, the maximal independent set of $P_m$ adjoined with the vertex $u_1$ of $C_n$ is $\frac{(m+1)}{2}$.  But the vertex $u_1$ is an element of the maximal independent set, $u_2$, the vertex adjacent to $u_1$ of $C_n$ cannot be in $I$.  So the maximal independent set from the path $P_m$ adjoined with the vertex $u_2$ consists of $\frac{(m-1)}{2}$ elements.  Proceeding like this, in all the paths of $G$, the maximal independent set corresponding to $u_i$, $1\leq i \leq n$ of the cycle $C_n$, is alternately $\frac{m+1}{2}$ and $\frac{m-1}{2}$.  Since there are $n$ vertices in $C_n$, the independence number of $G$, $\alpha(G)=\frac{n}{2} \lbrack  \frac{(m-1)}{2}+\frac{(m+1)}{2}\rbrack = \frac{nm}{2}$.   Let $E_1=\{ e_1, e_2, ...e_n\}$ be the edge set of the cycle $C_n$ and let $E_2=\{e_i^1, e_i^2, ... e_i^m\}$ be the edge set of the path $P_m$.  Now, for every $u_i$ in $C_n$, there exists a path with $m$ number of edges.  Clearly, $\frac{m}{2}$ edges contributes to a maximal matching of $P_m$.   For each $n$ vertices $u_i$, $1\leq i \leq n$  of $C_n$, there is a path $P_m$ adjoined to it, so that the maximal matching in $G$, $\nu(G)= \frac{n}{2}+n\lbrack\frac{m-1}{2}\rbrack = \frac{nm}{2}$.
\\Therefore, $\alpha(G)+\alpha(L(G))= \frac{nm}{2}+ \frac{nm}{2} = nm$ and
$\alpha(G).\alpha(L(G))= \frac{nm}{2}. \frac{nm}{2} = \frac{n^2m^2}{4}$.
\\\\
Case - 3 : ($m$ is even and $n$ is odd)
\\ Since $m$ is even,  and since $u_1^1, u_1^2,....u_1^m$ be the vertices of the first path attached to the first vertex $u_1$ of the cycle $C_n$, the maximal independent set consists of exactly $\frac{m}{2}$ elements.  Since there are $n$ number of paths attached to every vertex $u_i$, $i\leq 1\leq n$ of the cycle $C_n$, which contributes $n\frac{m}{2}$ number of elements to the maximal independent set $I$. Therefore $\alpha(G)=\frac{nm}{2}$. Let $E_1=\{ e_1, e_2, ...e_n\}$ be the edge set of the cycle $C_n$ and let $E_2=\{e_i^1, e_i^2, ... e_i^m\}$ be the edge set of the path $P_m$.  Now, for every $u_i$ in $C_n$, there exists a path with $m$ number of edges.  Clearly, $\frac{m}{2}$ edges contributes to a maximal matching of $P_m$.  For all $n$ vertices of the cycle $C_n$, there adjoined paths $P_m$, so that the maximal matching in $G$ is $\nu(G)=\frac{nm}{2}$.
\\Therefore $\alpha(G)+\alpha(L(G))= \frac{nm}{2}+ \frac{nm}{2} = nm$ and
$\alpha(G).\alpha(L(G))= \frac{nm}{2}. \frac{nm}{2} = \frac{n^2m^2}{4}$.
\\\\
Case - 4 : ($m$ is odd and $n$ is odd)
\\ Since $m$ is odd, the maximal independent set $I$ of $P_m$ adjoined with the vertex $u_1$ of $C_n$ is $\frac{(m+1)}{2}$.  But the vertex $u_1$ is an element of the maximal independent set of $P_m$, $u_2$, the vertex adjacent to $u_1$ of $C_n$ cannot be in $I$.  So the maximal independent set from the path $P_m$ adjoined with the vertex $u_2$ consists of $\frac{(m-1)}{2}$ elements.  Proceeding like this, in all the paths of $G$, the maximal independent set corresponding to $u_i$, $1\leq i \leq n$ of the cycle $C_n$, is alternately $\frac{m+1}{2}$ and $\frac{m-1}{2}$.  Obviously, there are $\lfloor\frac{n}{2}\rfloor$ number of paths which contributes $\frac{m+1}{2}$ and $\lceil\frac{n}{2}\rceil$ number of paths which contributes $\frac{m-1}{2}$ number of vertices to the maximal independend set, the independence number of $G$, $\alpha(G)= \lfloor \frac{n}{2}\rfloor \lbrack \frac{m+1}{2}\rbrack+ \lceil \frac{n}{2}\rceil \lbrack \frac{m-1}{2}\rbrack$.  Let $E_1=\{ e_1, e_2, ...e_n\}$ be the edge set of the cycle $C_n$ and let $E_2=\{e_i^1, e_i^2, ... e_i^m\}$ be the edge set of the path $P_m$.  Now, for every $u_i$ in $C_n$, there exists a path with $m$ number of edges.  Clearly, since $m$ is odd, $n \lbrack\frac{m-1}{2}\rbrack$ edges contributes to a maximal matching of $P_m$ for all paths of $G$.  Also, since $n$ is also odd, which contributes $\lfloor\frac{n}{2}\rfloor$ number of edges to the maximal matching of $G$, so that the independence number of $L(G)$ is equal to the matching number of $G$ is $\nu(G)= n \lbrack\frac{m-1}{2}\rbrack+\lfloor\frac{n}{2}\rfloor$.
\\
Therefore,\\
\begin{eqnarray*}
\alpha(G)+\alpha(L(G))&=& \lfloor\frac{n}{2}\rfloor \lbrack \frac{m+1}{2}\rbrack + \lceil \frac{n}{2}
\rceil \lbrack\frac{m-1}{2}\rbrack + \lfloor\frac{n}{2}\rfloor + n \lbrack\frac{m-1}{2}\rbrack \\
&=& \lfloor\frac{n}{2}\rfloor \lbrack \frac{m+1}{2}+1\rbrack + \lbrack\frac{m-1}{2}\rbrack \lbrack n+\lceil\frac{n}{2}\rceil\rbrack 
\end{eqnarray*} 
 and
\begin{eqnarray*}
\alpha(G).\alpha(L(G)) & = & \lbrack\lfloor\frac{n}{2}\rfloor\lbrack\frac{m+1}{2}\rbrack+\lceil\frac{n}{2}\rceil\lbrack\frac{m-1}{2}\rbrack\rbrack.\lbrack\lfloor\frac{n}{2}\rfloor + n\lbrack\frac{m-1}{2}\rbrack\rbrack
\end{eqnarray*}

\end{proof}
\section{Conclusion}{\rm
The theoretical results obtained in this research may provide a better insight into the problems involving matching number and independence number by improving the known lower and upper bounds on sums and products of independence numbers of a graph $G$ and an associated graph of $G$.  More properties and characteristics of operations on independence number and also other graph parameters are yet to be investigated.  The problems of establishing the inequalities on sums and products of independence numbers for various graphs and graph classes still remain unsettled. All these facts highlight a wide scope for further studies in this area.
\\\\This work is motivated by the inspiring talk given by Dr. J Paulraj Joseph, Department of Mathematics, Manonmaniam Sundaranar University,  TamilNadu, India titled \textbf{Bounds on sum of graph parameters - A survey}, at the National Conference on Emerging Trends in Graph Connections (NCETGC-2014), University of Kerala,  Kerala, India.}


\end{document}